\documentclass[a4paper,10pt]{amsart}

\usepackage{amsrefs}
\usepackage{amssymb}
\usepackage{amsmath}
\usepackage{amsthm}
\usepackage{color}
\usepackage{enumerate}
\input arrow

\DeclareMathOperator{\im}{Im} 

\newcommand{\modr}{\textrm{Mod-}R}
\newcommand{\rmod}{R\textrm{-Mod}}

\newcommand{\ext}{\textrm{Ext}_R}
\newcommand{\tor}{\textrm{Tor}^R}

\newcommand{\Coker}{\textrm{Coker }}
\newcommand{\filt}{\textrm{Filt-}}

\newtheorem{definition}{Definition}[section]
\newtheorem{proposition}[definition]{Proposition}
\newtheorem{theorem}[definition]{Theorem}

\newtheorem{examples}[definition]{Examples}
\newtheorem{lemma}[definition]{Lemma}
\newtheorem{corollary}[definition]{Corollary}
\newtheorem{remark}[definition]{Remark}

\title[PRODUCT OF FLATS AND MITTAG-LEFFLER DIMENSION]{PRODUCTS OF FLAT MODULES AND GLOBAL DIMENSION RELATIVE TO
  $\mathcal F$-MITTAG-LEFFLER MODULES}

\author{Manuel Cort\'es-Izurdiaga}
\address{Department of Mathematics, University of Almeria, E-04071,
  Almeria, Spain}
\thanks{Part of this paper was written while the author was visiting the
  School of Mathematics at the University of Manchester. The author is
  very grateful to Mike Prest for his hospitality and for many
  interesting discussions on the subject.\\
Partially supported by research project MTM-2014-54439 and by research
group ``Categor\'{\i}as, computaci\'on y teor\'{\i}a de anillos'' (FQM211) of
the University of Almer\'{\i}a} 
\email{mizurdia@ual.es}


\begin{document}

\maketitle

\begin{abstract}
  Let $R$ be any ring. We prove that all direct products of flat right
  $R$-modules have finite flat dimension if and only if each finitely
  generated left ideal of $R$ has finite projective dimension relative
  to the class of all $\mathcal F$-Mittag-Leffler left $R$-modules, where $\mathcal F$
  is the class of all flat right $R$-modules. In order to prove this theorem, we
  obtain a general result concerning global relative
  dimension. Namely, if $\mathcal X$ is any class of left
  $R$-modules closed under filtrations that contains all projective
  modules, then $R$ has finite left global projective dimension
  relative to $\mathcal X$ if and only if each left ideal of $R$ has
  finite projective dimension relative to $\mathcal X$. This result
  contains, as particular cases, the well known results concerning the
  classical left global, weak and Gorenstein global dimensions.
\end{abstract}

\section{INTRODUCTION}
\label{sec:introduction}

The motivation of this work comes from the study of right Gorenstein
regular rings.  A (non necessarily commutative with unit) ring $R$ is
said to be right Gorenstein regular if the category of right $R$-modules
is a Gorenstein category in the sense of \cite[Definition
2.18]{EnochsEstradaGarciaRozas}. These rings are precisely, by
\cite[Theorem 2.28]{EnochsEstradaGarciaRozas} and \cite[Theorem
1.1]{BennisMahdou}, those for which the right global Gorenstein
dimension is finite. Classical Iwanaga-Gorenstein rings, that is, two
sided noetherian rings with left and right self-injective dimensions
finite, are left and right regular Gorenstein. Actually, the right
Gorenstein regular property can be viewed as the natural one-sided
generalization of the Iwanaga-Gorenstein condition to non-noetherian
rings. Right Gorenstein regular rings have been studied in
\cite{EnochsEstradaGarciaRozas}, \cite{BeligiannisReiten},
\cite{EnochsIacobJenda}, \cite{EnochsEstradaIacob} and
\cite{CortesEnochsTorrecillas}.

In \cite[Corollary VII.2.6]{BeligiannisReiten} it is proved that a
ring $R$ is right Gorenstein regular if and only if the class of all
right $R$-modules with finite projective dimension coincides with the
class of all right $R$-modules with finite injective dimension. A direct consequence of this
fact is that the class of all modules with finite projective dimension
is closed under direct products. As it is deduced from our Theorem
\ref{t:ProdsOfFiniteFlats}, rings with this property satisfy that
direct products of right $R$-modules with finite flat dimension have
finite flat dimension. So, in order to understand right regular
Gorenstein rings it is necessary to study rings in which the class of
right $R$-modules with finite flat dimension is closed under direct products.

The main objective of this paper is to characterize rings for which
products of right modules with finite flat dimension have finite flat
dimension (we shall call them left weak coherent rings, since, by the
classical result \cite[Theorem 2.1]{chase}, they are natural
extensions of left coherent rings). This is essentially done in
Theorem \ref{t:ProdsOfFiniteFlats} where we characterize left weak
coherent rings as those which have finite left global dimension with
respect to the $\mathcal F$-Mittag-Leffler modules, where $\mathcal F$
is the class of all flat right $R$-modules (see Definition \ref{d:ML}
for details). At this point, a question arises naturally: is it
possible to obtain an intrinsic description of these rings? In order
to do this we study, in Section 3, the projective global dimension
relative to a class, $\mathcal X$, consisting of left $R$-modules that
contains all projective modules. The main result in this section,
Corollary \ref{c:GlobalDimensionIdeal}, states that (when $\mathcal X$
is closed under filtrations), the left global
projective dimension relative to $\mathcal X$ is finite if and only if
every left ideal has finite projective dimension relative of $\mathcal
X$. Moreover, if $\mathcal X$ is closed under direct limits (more
generally, under $(\aleph_0,\mathcal X)$-separable modules, see
Definition \ref{d:LocalFree} for details), the left global $\mathcal
X$-projective dimension is finite if and only if each finitely
generated left ideal has finite $\mathcal X$-projective
dimension. This result is interesting in its own because it contains,
as particular cases, well known theorems concerning the left global and weak
dimensions, see \cite[Theorem 1]{Auslander} and \cite[Theorem
8.25]{Rotman}, and the left Gorenstein global dimension, see
\cite[Proposition 3.5]{EnochsIacobJenda}.

This result about relative homological algebra allow us to improve the mentioned
characterization of left weak coherent rings, since, in order to compute the left global
dimension relative to the $\mathcal F$-Mittag-Leffler modules we only
have to look at finitely generated left ideals. Thus, in Theorem
\ref{t:ProdsOfFiniteFlats} we characterize left weak coherent rings
as those for which each finitely generated left ideal has finite
$\mathcal F$-Mittag-Leffler dimension. 

Let us point up that, as a byproduct of our
results, we obtain that left coherent rings are precisely those rings
with left $\mathcal F$-Mittag-Leffler global dimension 1: that is,
rings in which submodules of $\mathcal F$-Mittag-Leffler left
$R$-modules are again $\mathcal F$-Mittag-Leffler (see Corollary
\ref{c:coherent}). In addition, we prove that over left coherent rings
the class of all right $R$-modules with flat dimension less than or equal to $n$
(for some natural number $n$) is definable (see Corollary
\ref{c:Definable}).

\section{PRELIMINARIES}
\label{sec:preliminaries}

We fix, for the rest of the paper, a non necessarily commutative ring
with identity $R$. We shall work in the categories $\rmod$ and $\modr$
consisting of all left and right $R$-modules respectively. Morphisms
will operate on the opposite side than scalars. This implies that if
$f:M \rightarrow N$ and $g:N \rightarrow L$ are morphisms in $\rmod$,
then its composition will be $fg$. We shall denote by ${_R}{\mathcal
  F}$ the class of all flat left $R$-modules and by $\mathcal F_R$ the
class of all flat right $R$-modules; we shall omit the subscript when
there is no possible confusion. Given an infinite regular cardinal $\lambda$
and $M$ a left $R$-module, we shall say that $M$ is
$\lambda^<$-generated if it has a generating system with less than
$\lambda$ generators. Moreover, $M$ is said to be $\lambda^<$-presented
if there exists an exact sequence $R^{(\gamma)} \rightarrow R^{(\mu)}
\rightarrow M \rightarrow 0$ with $\mu$ and $\gamma$ cardinals smaller
than $\lambda$. If $\mathcal X$ is a class of left $R$-modules, we
shall denote by $\mathcal X^{< \lambda}$ the class of all
$\lambda^<$-presented modules belonging to $\mathcal X$. For any set $A$, $|A|$
will be its cardinality. We shall denote by $\omega$ the set of
natural numbers.

A direct system of left $R$-modules, $(M_i,\tau_{ij})_{i, j \in I}$, is
called well ordered if $I$ is an ordinal. Let $\kappa$ be an ordinal
and $(M_\alpha,\tau_{\alpha\beta})_{\alpha<\beta<\kappa}$ a well
ordered system of left $R$-modules. Following \cite[Definition
3.1.1]{GobelTrlifaj}, we shall say that the direct system is
continuous if for each limit ordinal $\beta < \kappa$ the direct limit
of the system $(M_\gamma,\tau_{\gamma\alpha})_{\gamma < \alpha <
  \beta}$ is the module $M_\beta$ with $\tau_{\alpha\beta}$ the
structural morphism from $M_\alpha$ to $M_\beta$ for each $\alpha <
\beta$. Let $\mathcal X$ be a class of left $R$-modules and suppose
that $\Coker \tau_{\alpha,\alpha+1} \in \mathcal X$ for each $\alpha+1
< \kappa$. If each $\tau_{\alpha \beta}$ is a monomorphism for every
pair $\alpha < \beta < \kappa$, we shall say that the direct system is
a $\mathcal X$-continuous direct system of monomorphisms. If,
moreover, each $\tau_{\alpha\beta}$ is an inclusion for every $\alpha
< \beta < \kappa$, we shall say that the direct system is a $\mathcal
X$-filtration. If $M$ is the direct limit of the filtration
$(M_{\alpha\beta},\tau_{\alpha\beta})_{\alpha < \beta < \kappa}$, we
say that $M$ is a $\mathcal X$-filtered module. Note that, in this
case, $M=\bigcup_{\alpha < \kappa}M_\alpha$ for a continuous chain of
submodules of $M$, $\{M_\alpha:\alpha < \kappa\}$, satisfying
$\frac{M_{\alpha+1}}{M_\alpha} \in \mathcal X$. We shall say that
$\mathcal X$ is closed under filtrations (resp. $\mathcal
X$-continuous limits of monomorphisms) if each $\mathcal X$-filtered
module (resp. each limit of a $\mathcal X$-continuous system of
monomorphisms) belongs to $\mathcal X$. It is easy to see that if $M$
is the limit of a $\mathcal X$-continuous system of monomorphisms then
it is a $\mathcal X$-filtered module. Then $\mathcal X$ is closed
under $\mathcal X$-continuous direct system of modules if and only if
it is closed under $\mathcal X$-filtrations.

The class $\mathcal X$ is called deconstructible (see \cite[Definition
1.4]{Stovicek}) if there exists a set of modules $\mathcal S$ of
$\mathcal X$ such that $\mathcal X$ is equal to
$\textrm{Filt-}\mathcal X$, the class of all $\mathcal S$-filtered
modules. As a consequence of \cite[Lemma 1.6]{Stovicek}, $\mathcal X$
is deconstructible if and only if it is closed under filtrations and
there exists a set $\mathcal S$ of $\mathcal X$ such that $\mathcal X
\subseteq \filt\mathcal S$.

Recall the notion of Mittag-Leffler with respect to a class of
modules, see \cite{Rothmaler} and \cite{AngeleriHerbera}.

\begin{definition}\label{d:ML}
  Let $\mathcal X$ be any class of right $R$-modules. A left
  $R$-module $M$ is said to be $\mathcal X$-Mittag-Leffler if for each
  family of right modules belonging to $\mathcal X$, $\{F_i:i \in
  I\}$, the canonical map from $\left(\prod_{i \in I}F_i\right)
  \otimes M$ to $\prod_{i \in I}F_i \otimes M$ is monic.
\end{definition}

For any class of right $R$-modules $\mathcal X$, we shall denote by
$\mathcal M^{\mathcal X}$ the class consisting of all $\mathcal
X$-Mittag-Leffler left $R$-modules. If $\mathcal X$ is equal to
$\modr$, then we get the classical definition of
Mittag-Leffler module, which was introduced by Raynaud and Gruson in
\cite{RaynaudGruson}.  We are specially interested in $\mathcal
F_R$-Mittag-Leffler left $R$-modules, which will be called Mittag-Leffler modules
with respect to the flat modules too. Of course, there are modules
which are Mittag-Leffler with respect to the flat modules but that are
not Mittag-Leffler. This follows from the facts that all left $R$-modules
are Mittag-Leffler if and only if $R$ is left pure-semisimple (see
\cite[Theorem 8]{AzumayaFacchini}), and that all left $R$-modules are
Mittag-Leffler with respect to the flats if and only if $R$ is left
noetherian (this follows easily from \cite[Theorem
1]{goodearl}). Then, if $R$ is left noetherian and not left
pure-semisimple, there exist modules in $\rmod$ which are
$\mathcal F$-Mittag-Leffler but not Mittag-Leffler.

By \cite[Theorem 2.6]{HerberaTrlifaj}, the Mittag-Leffler property can be
characterized in terms some local freenes. Let us recall what this
local property means:

\begin{definition}\label{d:LocalFree}
  Let $\mathcal X$ be a class of left $R$-modules and $\kappa$ an
  infinite regular cardinal. We shall say that a left $R$-module $M$ is
  $(\kappa,\mathcal X)$-free if it has a $(\kappa,\mathcal X)$-dense
  system of submodules, that is, a direct family of submodules,
  $\mathcal S$, such that:

  \begin{enumerate}
  \item $\mathcal S$ is closed under well-ordered ascending chains of
    length smaller than $\kappa$, and

  \item every subset of $M$ of cardinality smaller than $\kappa$ is
    contained in an element of $\mathcal S$.
  \end{enumerate}
\end{definition}

Given any class $\mathcal X$ of left $R$-modules, a module $M$ is
$(\aleph_0,\mathcal X)$-free if and only if it is the union of a
direct family of submodules belonging to $\mathcal X$. If $\mathcal
X$ is in $\modr$, it is proved in \cite[Theorem 2.6]{HerberaTrlifaj} that
the $\mathcal X$-Mittag-Leffler modules are precisely the
$(\aleph_1,\mathcal M^{\mathcal X})$-free modules.

Finally, let us fix some notation concerning the category of short
exact sequences of left $R$-modules, $\mathbf{Ses}(\rmod)$. Recall
that a morphism in this category is a triple $(f^1,f^2,f^3)$ making
the corresponding diagrams commutative. Given three classes, $\mathcal
X_1$, $\mathcal X_2$ and $\mathcal X_3$, of left $R$-modules, we shall
denote by $\textbf{Ses}(\mathcal X_1,\mathcal X_2,\mathcal X_3)$ the
full subcategory of $\textbf{Ses}(\rmod)$ consisting of all short
exact sequences with first term in $\mathcal X_1$, second term in
$\mathcal X_2$ and third term in $\mathcal X_3$. We shall use
continuous limits and filtrations in the category
$\textbf{Ses}(\rmod)$, which are defined in the same way as in the
category of $\rmod$.

\section{GLOBAL DIMENSION RELATIVE TO A CLASS}
\label{sec:glob-dimens-relat}

In this section we prove that the ring $R$ has finite left projective
global dimension relative to a class $\mathcal X$ which is closed
under filtrations and contains all projective modules if and only if
each left ideal of $R$ has finite projective dimension relative to
$\mathcal X$. As we mentioned in the introduction, this result is
crucial to characterize when products of right modules with finite flat
dimension have finite flat dimension.

\begin{definition}
  Let $\mathcal X$ be a class of left $R$-modules containing all
  projective modules.
  \begin{enumerate}
  \item Given a natural number $n$ and a left $R$-module $M$, we shall
    say that $M$ has projective dimension relative to $\mathcal X$ (or
    $\mathcal X$-projective dimension) less than or equal to $n$ if there
    exists projective resolution of $M$ such that its $(n-1)st$ syzygy
    belongs to $\mathcal X$. We shall denote by $\mathcal X_n$ the
    class of all modules with $\mathcal X$-projective dimension less than
    or equal to $n$.

  \item The left global projective dimension relative to $\mathcal X$
    (or left $\mathcal X$-projective global dimension) of $R$ is the
    supremum of the set consisting of the $\mathcal X$-projective
    dimensions of all left $R$-modules, if this supremum exists, and
    $\infty$ otherwise.
  \end{enumerate}
\end{definition}

Let $\mathcal X$ be any class of left $R$-modules that contains all
projective modules. It is very easy to prove that $\left(\mathcal
  X_n\right)_1 = \mathcal X_{n+1}$ for each natural number
$n$. Consequently, a left $R$-module $M$ has $\mathcal X$-projective dimension
less than or equal to $n+1$ if and only if there exists a short exact
sequence
\begin{displaymath}
  0 \mapright K \mapright P \mapright M \mapright 0
\end{displaymath}
with $P$ projective and $K$ with $\mathcal X$-projective dimension
less than or equal to $n$. If, in addition, $\mathcal X$ is closed under
direct summands and finite direct sums, the $\mathcal X$-projective
dimension does not depend on the chosen projective resolution, since,
for each natural number $n$, any two $n$-sysygies of a module are
projectively equivalent by \cite[Proposition 8.5]{Rotman}. Finally, if
$\mathcal X$ is the left hand class of a hereditary cotorsion pair
cogenerated by a set, then, by \cite[Proposition
1.11]{EstradaGuilIzurdiaga}, the $\mathcal X$-resolution dimension can
be computed using exact sequences with terms in $\mathcal X$ or left
$\mathcal X$-resolutions in the sense of \cite[Definition
8.1.2]{EnochsJenda}. Moreover, in this case, the $\mathcal
X$-projective dimension of a module $M$ is the least natural number
$n$ such that $\ext^n(M,C) = 0$ for each $C \in \mathcal X^\perp$.

Let $\mathcal X$ be any class of modules. In the following result we establish that some closure properties of $\mathcal X$ are
inherited by $\mathcal X_n$. These properties will be useful to
compute global dimensions.

\begin{proposition}\label{p:ClosedUnderFiltrations}
  Let $\mathcal X$ be a class of left $R$-modules that contains all
  projective modules. Then
  \begin{enumerate}
  \item If $\mathcal X$ is closed under filtrations, then $\mathcal
    X_n$ is closed under filtrations for each $n \in \omega$.

  \item If $\mathcal X$ is deconstructible, then $\mathcal X_n$ is
    deconstructible for each $n \in \omega$.

  \item Suppose that $\mathcal X$ is closed under direct summands and
    finite direct sums. Let $\kappa$ be any infinite regular cardinal such
    that each $(\kappa, \mathcal X)$-free module belongs to $\mathcal
    X$. Then, each $(\kappa,\mathcal X_n)$-free module belongs to
    $\mathcal X_n$ for every $n \in \omega$.
  \end{enumerate}
\end{proposition}

\begin{proof}
  (1) We shall induct on $n$. Case $n=0$ is true by
  hypothesis. Suppose that we have proven the result for some natural
  number $n$ and let us prove it for $n+1$. Let $M$ be an $\mathcal
  X_{n+1}$-filtered module and let $\{M_\alpha:\alpha < \kappa\}$ be a
  $\mathcal X_{n+1}$-filtration of $M$ for some ordinal
  $\kappa$. Take, for each $\alpha < \kappa$, a short exact sequence
  \begin{displaymath}
    \mathbf{\overline S}_{\alpha+1}: \quad 0 \mapright \overline
    X_{\alpha+1} \mapright^{\overline f_{\alpha+1}} \overline P_{\alpha
      +1} \mapright^{\overline g_{\alpha+1}} \frac{M_{\alpha+1}}{M_\alpha} \mapright 0
  \end{displaymath}
  with $\overline P_{\alpha + 1}$ projective and $\overline
  X_{\alpha+1} \in \mathcal X_n$. We are going to construct a
  $\textbf{Ses}(\mathcal X_n,\mathcal P, \rmod)$-continuous direct
  system in $\textbf{Ses}(\rmod)$, $(\mathbf S_\alpha, \tau_{\alpha
    \beta})_{\alpha \leq \beta < \kappa}$, with
  \begin{displaymath}
    \mathbf{S}_\alpha: \quad    0 \mapright X_\alpha \mapright^{f_\alpha} P_\alpha \mapright^{g_\alpha} M_\alpha
    \mapright 0,
  \end{displaymath}
  satisfying:
  \begin{enumerate}[(i)]
  \item $\tau_{\alpha,\alpha+1}^1:X_\alpha \rightarrow X_{\alpha+1}$ is a monomorphism,
    $\tau_{\alpha,\alpha+1}^2:P_\alpha\rightarrow P_{\alpha+1}$ is a
    split monomorphism and $\tau_{\alpha,\alpha+1}^3:M_\alpha
    \rightarrow M_{\alpha+1}$ is the inclusion.

  \item For each $\beta < \kappa$ limit, $\displaystyle \mathbf
    S_\beta = \lim_{\substack{\longrightarrow\\\alpha < \beta}}\mathbf
    S_\alpha$.
  \end{enumerate}

  These $S_\alpha$ and $\tau_{\gamma \alpha}$, for each $\gamma <
  \alpha < \kappa$, can be constructed recursively. Let us sketch this
  construction. Case $\alpha = 0$ is trivial. Given any $\alpha$ such
  that $\mathbf S_\alpha$ and $\tau_{\gamma \alpha}$ have been already
  constructed for each $\gamma < \alpha$, we can construct, as in the
  proof of the Horseshoe lemma, \cite[Proposition 6.24]{Rotman}, a
  commutative diagram with exact rows and columns
  \begin{displaymath}
    \setlength{\harrowlength}{20pt}
    \setlength{\varrowlength}{20pt}
    \commdiag{& & 0 & & 0 & & 0 & & \cr
      & & \mapdown & & \mapdown & & \mapdown & & \cr
      0 & \mapright & X_\alpha & \mapright^{\tau^1_{\alpha\alpha+1}} & X_{\alpha+1} &
      \mapright & \overline X_{\alpha+1} & \mapright & 0\cr
      & & \mapdown & & \mapdown & & \mapdown & & \cr
      0 & \mapright & P_\alpha & \mapright^{\tau^2_{\alpha\alpha+1}} & P_\alpha \oplus \overline P_{\alpha+1} &
      \mapright & \overline{P}_{\alpha+1} & \mapright & 0\cr
      & & \mapdown & & \mapdown & & \mapdown & & \cr
      0 & \mapright & M_\alpha & \mapright_{\tau^3_{\alpha\alpha+1}} & M_{\alpha+1} &
      \mapright & \frac{M_{\alpha+1}}{M_\alpha} & \mapright & 0\cr
      & & \mapdown & & \mapdown & & \mapdown & & \cr
      & & 0 & & 0 & & 0 & & \cr}
  \end{displaymath}

  Then take $\mathcal S_{\alpha + 1}$ to be the middle column of the
  diagram with $P_{\alpha+1} = P_\alpha \oplus
  \overline{P}_{\alpha+1}$. If $\alpha$ is a limit ordinal and we have
  already constructed the sequence $\mathbf S_\gamma$ for each $\gamma
  < \alpha$, then take $\mathbf S_\alpha$ to be the direct limit of
  the system $(\mathbf S_\gamma,\tau_{\gamma \beta})_{\gamma \leq
    \beta < \alpha}$. This finishes the construction.

  Now let
  \begin{displaymath}
    \mathbf{ S}: \quad 0 \mapright 
    X \mapright P \mapright N \mapright 0
  \end{displaymath}
  be the direct limit of the system $(\mathbf
  S_\alpha,\tau_{\alpha\beta})_{\alpha < \beta < \kappa}$. Since
  direct limits in the category of short exact sequences are computed
  componentwise, we get that $N = M$. Moreover, we conclude that $X$
  is a $\mathcal X_n$-filtered module and, by induction hypothesis, it
  belongs to $\mathcal X_{n}$; and that $P$, being filtered by
  projective modules, is projective. Then the short exact sequence
  $\mathbf S$ says that $M \in \mathcal X_{n+1}$ and the proof is
  finished.

  (2) If $\mathcal X$ is deconstructible then $\mathcal X_n$ is closed
  under $\mathcal X_n$-filtrations by (1). Moreover, if there exists a
  set $\mathcal S$ such that each module in $\mathcal X$ is $\mathcal
  S$-filtered then, actually there exists an infinite regular cardinal
  number $\lambda$ such that $\mathcal X \subseteq \filt \mathcal X^{<
    \lambda}$. Then, the same proof of \cite[Theorem
  2.2]{EstradaGuilIzurdiaga} gives that $\mathcal X_n \subseteq \filt
  \mathcal X_n^{< \lambda}$ (the mentioned result is proved when
  $\mathcal X$ is the left hand class of a hereditary cotorsion pair
  cogenerated by a set; but in the proof it is only used that
  $\mathcal X \subseteq \filt \mathcal X^{< \lambda}$ for some
  infinite regular cardinal).

  (3) In some part of the following proof we have to assume that $R$
  has more than two elements. Note that if $R$ has two elements there
  is nothing to prove as $\mathcal X_n = \mathcal X$ for each $n \in
  \omega$.

  We shall induct on $n$. Case $n=0$ is the hypothesis. Suppose that
  the result is true for any natural number $n$ and let us prove it
  for $n+1$. Let $M$ be any $(\kappa,\mathcal X_{n+1})$-free module
  and fix $\mathcal S$ a $(\kappa,\mathcal X_{n+1})$-dense system of
  $M$. For each $m \in M$ let $\varphi_m:R \rightarrow Rm$ be the
  morphism given by multiplication; for each $T \leq M$, denote by
  $\varphi_T:R^{(T)} \rightarrow M$ the induced morphism and by $K_T$
  its kernel. Let us denote by $e_m$ the canonical element of
  $R^{(M)}$ for each $m \in M$. We claim that $\mathcal T = \{K_S:S
  \in \mathcal S\}$ is a $(\kappa,\mathcal X)$-dense system of
  $K_M$. First of all note that, since the $\mathcal X$-dimension does not
  depend on the chosen projective resolution (as $\mathcal X$ is
  closed under direct summands and finite direct sums), $K_{
    S} \in \mathcal X_n$ for each $ S \in \mathcal
  S$. Moreover, $\mathcal T$ is trivially a direct system and each $X
  \leq K_M$ with $|X| < \kappa$ is contained in some $K_S$ for some $S
  \in \mathcal S$.

  It only remains to prove that $\mathcal T$ is closed under
  well-ordered ascending chains of length smaller than $\kappa$. First
  of all, let us prove that if $K_S \leq K_T$ for some $S,T \in
  \mathcal S$, then $S \leq T$. Let $m \in S$ and suppose that $m
  \notin T$ and take $r \in R-\{1\}$. Since $re_m-e_{rm}$ and
  $(1-r)e_m-e_{(1-r)m}$ are nonzero element in $K_S \leq K_T$ and $m
  \notin T$, we conclude that both $rm$ and $(1-r)m$ belong to $T$
  (note that, if $L \in \mathcal S$ and $l,t \notin L$ then $K_L \cap
  (R_l\oplus R_t) = 0$, $R_l$ and $R_t$ being the copies of $R$ in
  coordinates $l$ and $s$ respectively). Then $m = rm+(1-r)m$ belongs
  to $T$, a contradiction. Now, using this fact, each well ordered chain
  $\{K_{S_\alpha}:\alpha < \mu\}$ of modules in $\mathcal T$ (with
  $\mu < \kappa$), gives, taking unions, the short exact sequence
  \begin{displaymath}
    0 \mapright \bigcup_{\alpha < \mu} K_{S_\alpha} \mapright R^{(S)}
    \mapright S \mapright 0
  \end{displaymath}
  where $S = \bigcup_{\alpha < \mu}S_\alpha$. Then, since
  $\bigcup_{\alpha < \mu}K_{S_\alpha} = K_S$ and $S \in \mathcal S$,
  we conclude that $\bigcup_{\alpha < \mu}K_{S_\alpha} \in \mathcal
  T$. Consequently, $\mathcal T$ is closed under well ordered unions.

  The conclusion is that $K_M$ is $(\kappa,\mathcal X_n)$-free and by
  induction hypothesis, $K_M \in \mathcal X_n$.  Consequently, $M \in
  \mathcal X_{n+1}$ and the proof is finished.
\end{proof}

\begin{remark}
  (1) of Proposition 3.2 extends \cite[Lemma
  2.1]{EstradaGuilIzurdiaga}, where it is proved that $\mathcal X_n$
  is closed under filtrations when $\mathcal X$ is the left hand class
  of a hereditary cotorsion pair cogenerated by a set.
\end{remark}

The preceding proposition contains well known results concerning the projective,
flat and Gorenstein projective dimensions. Moreover, we can extend the
characterization of Mittag-Leffler modules given in \cite[Theorem
2.6]{HerberaTrlifaj} to modules with finite dimension with respect to
Mittag-Leffler modules.

\begin{corollary}
  Let $\mathcal X$ be any class of right $R$-modules and $M$ a
  module. The following assertions are equivalent:
  \begin{enumerate}
  \item $M$ has $\mathcal M^{\mathcal X}$-Mittag-Leffler dimension
    less than or equal to $n$.

  \item $M$ is $(\aleph_1,\mathcal M^{\mathcal X}_n)$-free.
  \end{enumerate}
\end{corollary}

Now we can prove the main theorem of this section:

\begin{theorem}
  Let $\lambda$ be an infinite regular cardinal, $n$, a nonzero
  natural number and $\mathcal X$ a class of left $R$-modules closed under
  filtrations that contains all projective modules. The following
  assertions are equivalent:
  \begin{enumerate}
  \item Each $\lambda$-presented module has $\mathcal X$-projective
    dimension less than or equal to $n$.

  \item Each $\lambda^<$-generated left ideal of $R$ has $\mathcal
    X$-projective dimension less than or equal to $n-1$.
  \end{enumerate}
\end{theorem}

\begin{proof}
  (1) $\Rightarrow$ (2) If $I$ is a left ideal of $R$ which is
  $\lambda^<$-generated, the module $\frac{R}{I}$ is
  $\lambda^<$-presented and, by hypothesis, has $\mathcal X$-resolution
  dimension less than or equal to $n$. This means that $I$ has $\mathcal
  X$-resolution dimension less than or equal to $n-1$.

  (2) $\Rightarrow$ (1). Let $M$ be any $\lambda^<$-presented module,
  choose $\mu$ a cardinal smaller than $\lambda$ such that there
  exists an epimophism $f:R^{(\mu)} \rightarrow M$ and that, its
  kernel $K$ is $\lambda^<$-generated. We claim that $K$ has $\mathcal
  X$-resolution dimension less than or equal to $n-1$.

  In order to prove the claim, denote by $p_\alpha:R^{(\mu)}
  \rightarrow R$ and $i_\alpha:R \rightarrow R^{(\mu)}$ the
  projection and inclusion for each $\alpha < \mu$, and
  by $K_\alpha = K \cap R^{(\alpha)}$. Note that the family
  $\{K_\alpha:\alpha < \kappa\}$ is a filtration of $K$. Now, for each
  $\alpha < \kappa$, we have the exact sequence
  \begin{displaymath}
    0 \mapright K_\alpha \mapright K_{\alpha+1}
    \mapright^{q_{\alpha}} R
  \end{displaymath}
  where $q_{\alpha}$ is the restriction of $p_\alpha$ to
  $K_{\alpha+1}$. Note that $\im q_{\alpha}$ is $\lambda^<$-generated: if
  $\{x_\delta:\delta < \beta\}$ is a generating set of $K$ with
  $\beta$ a cardinal smaller than $\lambda$, then
  $\{(x_\delta)p_\gamma i_\gamma:\delta < \beta, \gamma < \mu\}$ is a
  generating set of $K$, from which follows that $\{(x_\delta)p_\gamma
  i_\gamma:\delta < \beta, \gamma < \alpha+1\}$ is a generating set of
  $K_{\alpha+1}$ with cardinality smaller than $\lambda$. The
  conclusion is that $\frac{K_{\alpha+1}}{K_\alpha}$ is isomorphic to
  a $\lambda^<$-generated left ideal of $R$ and, by hypothesis, has
  $\mathcal X$-resolution dimension less than or equal to $n-1$. This
  means that $K$ is $\mathcal X_{n-1}$-filtered and by Proposition
  \ref{p:ClosedUnderFiltrations} belongs to $\mathcal X_{n-1}$. This
  proves our claim and finishes the proof of the theorem.
\end{proof}

Note that always exists an infinite regular cardinal $\lambda$ such
that $R$ is left $\lambda$-noetherian, in the sense that each left
ideal is $\lambda^<$-generated. Then, as an immediate consequence of
the previous result, if each $\lambda^<$-presented module has finite
$\mathcal X$-projective dimension, then $R$ has finite left global
$\mathcal X$-projective dimension. Moreover, if we do not establish any cardinal restriction, we obtain an intrinsic
description of rings with finite global projective relative dimension:

\begin{corollary}\label{c:GlobalDimensionIdeal}
  Let $n$ be a nonzero natural number and $\mathcal X$ a class of left
  $R$-modules closed under filtrations that contains all projective
  modules. The following assertions are equivalent:
  \begin{enumerate}
  \item $R$ has left global $\mathcal X$-projective dimension less than or
    equal to $n$.

  \item Each left ideal of $R$ has $\mathcal X$-projective dimension
    less than or equal to $n-1$.
  \end{enumerate}

  If, in addition, $\mathcal X$ is closed under direct summands,
  finite direct sums and each $(\kappa,\mathcal X)$-free module
  belongs to $\mathcal X$, for some infinite regular cardinal $\kappa$, then these
  conditions are equivalent to:

\begin{enumerate}
  \setcounter{enumi}{2}
\item Each $\kappa^<$-generated left ideal of $R$ has $\mathcal
  X$-projective dimension less than or equal to $n-1$.
\end{enumerate}

\begin{proof}
  (1) $\Leftrightarrow$ (2) follows from the preceding theorem. (2)
  $\Rightarrow$ (3) is clear. In order to prove (3) $\Rightarrow$ (2)
  simply note that, if $I$ is a left ideal, the hypothesis says that
  the set of all $\kappa^<$-generated left ideals contained in $I$ is
  a $(\kappa,\mathcal X_{n-1})$-dense system of $I$. By Proposition
  \ref{p:ClosedUnderFiltrations}, $I$ has $\mathcal X$-dimension less than
  or equal to $n-1$.
\end{proof}
\end{corollary}

This result contains, as particular cases, the well known results
concerning the global, weak and Gorenstein global dimensions (see
\cite[Theorem 8.16]{Rotman}, \cite[Theorem 8.25]{Rotman} and
\cite[Proposition 3.5]{EnochsIacobJenda}). Moreover, it can be applied
to the $\mathcal F$-Mittag-Leffler dimension.

\begin{corollary}\label{c:Mittag-LefflerDim}
  Let $n$ be a nonzero natural number. Then $R$ has left $\mathcal
  M^{\mathcal F}$-projective global dimension less than or equal to $n$
  if and only if each finitely generated left ideal has $\mathcal
  M^{\mathcal F}$-projective dimension less than or equal to $n-1$.
\end{corollary}

\begin{proof}
  Follows from the preceding result using that the class $\mathcal
  M^{\mathcal F}$ is closed under filtrations by \cite[Proposition
  1.9]{AngeleriHerbera} and under $(\aleph_0,\mathcal M^{\mathcal
    F})$-free modules, that is, under direct unions
  of submodules, by \cite[Theorem 1]{goodearl}.
\end{proof}

Another application of our result is when the left global and weak
dimensions coincide (for example, when the ring is left noetherian or
left perfect). In this case, the left global dimension is the supremum
of the projective dimensions of the finitely generated left ideals:

\begin{corollary}
  Let $R$ be a ring such that the left global and weak dimensions
  coincide, and let $n$ be a nonzero natural number. Then $R$ has left global
  dimension less than or equal to $n$ if and only if each finitely
  generated left ideal has projective dimension less than or equal to $n-1$.
\end{corollary}

\begin{proof}
  Simply note that by hypothesis and Corollary
  \ref{c:GlobalDimensionIdeal}, the left global dimension is equal to
  the supremum of the flat dimensions of all finitely generated left
  ideals. But this supremum is smaller or equal than the supremum of
  the projective dimension of all finitely generated left ideals.
\end{proof}

Let us finish this section with an useful result for computing
global dimensions. It is is proven with the argument used in
\cite[Corollary VII.2.6]{BeligiannisReiten}.

\begin{lemma}\label{l:FiniteGlobalDimension}
  Let $\mathcal X$ and $\mathcal Y$ be classes of left $R$-modules
  such that $\mathcal X$ is closed under direct summands, finite
  direct sums and contains all projective modules, and $\mathcal Y$ is
  closed under countable direct sums or
  countable direct products. Then the following assertions are
  equivalent:
  \begin{enumerate}
  \item Each module in $\mathcal Y$ has finite $\mathcal X$-projective
    dimension.

  \item There exists a natural number $n$ such that each module in
    $\mathcal Y$ has $\mathcal X$-projective dimension less than or equal
    to $n$.
  \end{enumerate}
\end{lemma}

\begin{proof}
  We only have to prove that (1) implies (2). We do it assuming that
  $\mathcal Y$ is closed under countable direct sums (the case of
  products is analogous). Suppose that (2) is false and take, for each
  natural number $n$, a module $Y_n$ in $\mathcal Y$ with $\mathcal
  X$-projective dimension equal to $n$. Then, $Y=\bigoplus_{n \in
    \omega}Y_n$ belongs to $\mathcal Y$ and has infinite $\mathcal
  X$-dimension since, otherwise, it would exist $m \in \omega$
  such that $Y$ would have $\mathcal X$-projective dimension less than or
  equal to $m$, which would imply that $Y_{m+1}$ would have this
  property too (note that if $\mathcal X$ is closed under direct
  summands, then so is $\mathcal X_m$). But the dimension of $Y_{m+1}$
  is $m+1$, which is a contradiction.
\end{proof}

\section{WHEN DIRECT PRODUCTS OF FLAT MODULES HAVE FINITE FLAT DIMENSION}

In this section we investigate rings for which direct products of
right $R$-modules with finite flat dimension have finite flat
dimension. As a consequence of our first result, we only have to look
at direct product of flat modules.

\begin{proposition}\label{p:BoundedDimensionOfProducts}
  The following assertions are equivalent:
  \begin{enumerate}
  \item The direct product of any family of flat right $R$-modules has
    finite flat dimension.

  \item There exists a natural number $n$ such that the product of any
    family of flat right $R$-modules has flat dimension less than or equal
    to $n$.

  \item There exists a natural number $m$ such that the direct product
    of any family of right $R$-modules with flat dimension less than or
    equal to $m$ has finite flat dimension.

\item There exist natural numbers $m$ and $n$ such that the direct
  product of any family of right $R$-modules with flat dimension less than
  or equal to $m$ has flat dimension less than or equal to $m+n$.

  \item For any natural number $m$, the direct product of any family
    of right $R$-modules with flat dimension less than or equal to $m$
    has finite flat dimension.

  \item For any natural number $m$ there exists a natural number $n$
    such that the direct product of any family of right $R$-modules
    with flat dimension less than or equal to $m$ has flat dimension
    less than
    or equal to $m+n$.
  \end{enumerate}
\end{proposition}

\begin{proof}
  (1) $\Leftrightarrow$ (2). Follows from Lemma
  \ref{l:FiniteGlobalDimension} by taking $\mathcal Y$ the class of
  all products of flat right $R$-modules and $\mathcal X$ the class of
  all flat right $R$-modules.

  (3) $\Leftrightarrow$ (4) and (5) $\Leftrightarrow$ (6). Again
  follows from Lemma \ref{l:FiniteGlobalDimension} by taking $\mathcal
  Y$ the class consisting of all products of right $R$-modules with
  flat dimension less than or equal to $m$ and $\mathcal X$ the class of
  all flat right $R$-modules.

  (1) $\Rightarrow$ (5). By induction on $m$. Case $m=0$ is (1). Suppose that the result is true for some natural
  number $m$ and let us prove it for $m+1$. Let $\{F_i:i \in I\}$ be a
  family of right $R$-modules with flat dimension less than or equal to
  $m+1$. Then, for each $i \in I$ there exists a projective
  presentation
  \begin{displaymath}
    0 \mapright K_i \mapright P_i \mapright F_i \mapright 0
  \end{displaymath}
  with $K_i$ having flat dimension less than or equal to $n$. With all
  these sequences, we can form the short exact sequence
  \begin{displaymath}
    0 \mapright \prod_{i \in I}K_i \mapright \prod_{i \in I}P_i
    \mapright \prod_{i \in I}F_i \mapright 0
  \end{displaymath}
  Now apply the induction hypothesis to get that $\prod_{i \in I}P_i$
  and $\prod_{i \in I}K_i$ have finite flat dimension and,
  consequently, that $\prod_{i \in I}F_i$ has finite flat dimension
  too.

  (5) $\Rightarrow$ (3) $\Rightarrow$ (1). Trivial.
\end{proof}

We shall say that $R$ is a left weak coherent ring if it satisfies the
equivalent conditions of the preceding result. Suppose that $R$ is
such a ring. Then we can take $n$ the maximum of the set consisting of
all flat dimensions of all direct products of flat right
$R$-modules. In this case, we shall say that $R$ is left weak
$n$-coherent. Note that left weak $0$-coherent rings are precisely
left coherent rings. In the following result we characterize left weak
coherent rings in terms of the left global $\mathcal M^{\mathcal
  F}$-projective resolution.

\begin{theorem}\label{t:ProdsOfFiniteFlats}
  Let $n$ be a natural number. The following assertions are equivalent:
  \begin{enumerate}
  \item $R$ is left weak $n$-coherent.

  \item There exists a set $I_0$ with $|I_0| \geq |R|$ such that
    $R_R^{I_0}$ has flat dimension less than or equal to $n$.

  \item $R$ has left global $\mathcal M^{\mathcal F}$-projective
    dimension less than or equal to $n+1$.

  \item Each finitely generated left ideal of $R$ has $\mathcal M^{\mathcal
      F}$-projective dimension less than or equal to $n$.

  \item Each cyclic left $R$-modules has $\mathcal M^{\mathcal
      F}$-projective dimension less than or equal to $n+1$.
  \end{enumerate}
\end{theorem}

\begin{proof}
  (1) $\Rightarrow$ (2). Trivial.

  (2) $\Leftrightarrow$ (3). Let $M$ be a left $R$-module. We claim
  that the $\mathcal M^{\mathcal F}$-projective dimension of $M$ is
  less than or equal to $n$ if and only if
  $\tor_{n}\left(R^{I_0},M\right)=0$. Then the equivalence of (2) and
  (3) follows from this claim.

  We proceed by induction on $n$. Fix a projective presentation of
  $M$,
  \begin{equation}\label{eq:a}
    0  \mapright  K \mapright  P  \mapright  M  \mapright  0
  \end{equation}
  In order to prove case $n=1$ consider the following commutative
  diagram with exact rows
  \begin{displaymath}
    \commdiag{ &  & R^{I_0} \otimes K &
      \mapright^f & R^{I_0} \otimes P &
      \mapright & R^{I_0} \otimes M &
      \mapright & 0\cr
      & & \mapdown\lft{g} & & \mapdown & & \mapdown & & \cr
      0 & \mapright & K^{I_0} & \mapright &
      P^{I_0} & \mapright & M^{I_0} & \mapright & 0}
  \end{displaymath}
  Since $\tor_1\left(R^{I_0},P\right)=0$, as $P$ is projective, we
  conclude that $\tor_1\left(R^{I_0},N\right)=0$ if and only if $f$ is
  monic, if and only if $g$ is monic if and only if $K$ is $\mathcal
  F$-Mittag-Leffler by \cite[Theorem 1]{goodearl}, if and only if $M$
  has $\mathcal M^{\mathcal F}$-projective dimension less than or equal
  to $1$.

  Now suppose that we have proved the claim for some nonzero natural
  number $n$. Then, looking at the exact sequence (\ref{eq:a}), we get
  that $\tor_{n+1}\left(R^{I_0},M\right)=0$ if and only if
  $\tor_{n}\left(R^{I_0},K\right)=0$, if and only if the $\mathcal
  M_\mathcal F$-resolution dimension of $K$ is less than or equal to $n$
  (by induction hypothesis), if and only if the $\mathcal M_{\mathcal
    F}$-resolution dimension of $M$ is less than or equal to $n+1$.

  (3) $\Leftrightarrow$ (4) $\Leftrightarrow$ (5). These are Corollary
  \ref{c:Mittag-LefflerDim}.
\end{proof}

When $n=0$ we obtain new characterizations of left coherent rings:

\begin{corollary}\label{c:coherent}
  The following assertions are equivalent:
  \begin{enumerate}
  \item $R$ is left coherent.

  \item For each natural number $m$, the direct product of any family
    of right $R$-modules with flat dimension less than or equal to $m$
    has flat dimension less than or equal to $m$.

  \item Each left $R$-module has $\mathcal M^{\mathcal F}$-projective
    dimension less than or equal to 1.

  \item Each left ideal of $R$ is $\mathcal F$-Mittag-Leffler.

  \item Each submodule of a $\mathcal F$-Mittag Leffler left
    $R$-module is $\mathcal F$-Mittag Leffler.
  \end{enumerate}
\end{corollary}

\begin{proof}
  (1) $\Leftrightarrow$ (2). If $R$ is left coherent, direct products
  of flat right $R$-modules are flat by Chase's result, \cite[Theorem
  2.1]{chase}. Then, the same proof as (1) $\Rightarrow$ (3) in
  Proposition \ref{p:BoundedDimensionOfProducts} gives the equivalence
  of (1) and (2).

  The equivalences of (1), (3) and (4) follow from the
  previous theorem. (5) $\Rightarrow$ (3) is clear. It remains to
  prove that (3) implies (5). Let $M$ be a $\mathcal F$-Mittag-Leffler
  left $R$-module and $K \leq M$ a submodule. Take a projective
  presentation $f:P \rightarrow \frac{M}{K}$, and, making pullback, we
  can construct the following commutative diagram with exact rows and
  columns:
  \begin{displaymath}
    \setlength{\harrowlength}{20pt}
    \setlength{\varrowlength}{20pt}
    \commdiag{ &  &  &  & 0 & & 0 & & \cr
      & & & & \mapup & & \mapup & & \cr
      0 & \mapright & K & \mapright & M & \mapright &
      \frac{M}{K} & \mapright & 0\cr
      & & \mapup & & \mapup & &
      \mapup & &\cr
      0 & \mapright & K & \mapright & Q & \mapright &
      P & \mapright & 0\cr
      & & & & \mapup & & \mapup & & \cr
      & & & & K' & \mapright & K' & & \cr
      & & & & \mapup & & \mapup & & \cr
      & & & & 0 & & 0 & & }
  \end{displaymath}
  By (3), $K'$ belongs to $\mathcal M^{\mathcal F}$ and by \cite[Corollary 2.4]{Rothmaler}, $Q$ belongs to $\mathcal
  M_{\mathcal F}$ too. Since the middle row is split exact, $K$
  belongs to $\mathcal M^{\mathcal F}$ again by \cite[Corollary
  2.4]{Rothmaler} and the proof is finished.
\end{proof}

This result says that rings in which the class $\mathcal M^{\mathcal
  F}$ is closed under submodules, that is, rings with left global
$\mathcal M^{\mathcal F}$-projective dimension equal to one are,
precisely, left coherent rings. Recall that, as it can be easily
deduced from \cite[Theorem 1]{goodearl}, rings with left global
$\mathcal M^{\mathcal F}$-projective dimension zero are left
noetherian rings.

\begin{examples}
  \begin{enumerate}
  \item There exist left weak coherent rings which are not left
    coherent. For example, let $R$ be a non semihereditary ring with
    weak dimension 1 (such a ring is constructed in \cite[Example
    3.1.2]{glaz}). Then $R$ weak 1-coherent, but it is not left
    coherent by \cite[Theorem 3.1.3]{glaz}.

  \item There exist left weak coherent rings which have infinite weak
    global dimension. For instance, in \cite[Example
    7.7.2]{McconnellRobson} it is given a noetherian ring which has
    infinite global dimension. Since, for noetherian rings, the global
    and weak dimensions coincide, this ring has infinite weak
    dimension too.

  \item Right regular Gorenstein rings are left weak coherent. This is
    because, as it has been proved in \cite[Corollary
    VII.2.6]{BeligiannisReiten}, for these rings the class of all
    right $R$-modules with finite projective dimension coincides with
    the class of all modules with finite injective dimension. In
    particular, this means that the class of all right $R$-modules
    with finite projective dimension is closed under direct
    products. Then all direct products of copies of $R_R$ have finite
    flat dimension. By Theorem
    \ref{t:ProdsOfFiniteFlats}, $R$ is left weak coherent.
  \end{enumerate}
\end{examples}

Let us finish the paper with some applications of our results. In the
first of them we characterize when the class of all right $R$-modules with finite
flat dimension is closed under products. Recall that the right finitistic
flat dimension is the supremum of the set consisting of the flat
dimensions of all right $R$-modules with finite flat dimension, if this
supremum exists, and $\infty$ otherwise.

\begin{corollary}
  The following assertions are equivalent:
  \begin{enumerate}
  \item The class of all right $R$-modules with finite flat dimension
    is closed under direct products.
  \item $R$ is left
    weak coherent and has finite right finitistic flat dimension.
  \end{enumerate}
\end{corollary}

\begin{proof}
  (1) $\Rightarrow$ (2). By Lemma \ref{l:FiniteGlobalDimension},
  taking $\mathcal X$ and $\mathcal Y$ the class of all flat modules.

 (2) $\Rightarrow$ (1). Trivial.
\end{proof}

Recall that a class of right $R$-modules is definable if it is closed
under products, pure submodules and direct limits. As a second
application of our results, we are going to prove that the class of
all right $R$-modules with flat dimension less than or equal to $n$ is
definable for every $n \in \omega$, when $R$ is a left coherent
ring. We shall use the following preliminary lemma that says that, in
order to see that a class is closed under direct limits, we only have
to prove that it is closed under continuous well ordered limits. The
proof can be done with the same argument used in \cite[Corollary
1.7]{AdamekRosicky}.

\begin{lemma}
  Let $\mathcal X$ be a class of right $R$-modules which is closed under
  continuous well ordered limits. Then $\mathcal X$ is closed under
  direct limits.
\end{lemma}

\begin{proposition}\label{p:DirectLimits}
  For each natural number $n$, the class of all right $R$-modules with
  flat dimension less than or equal to $n$ is closed
  under direct limits and pure submodules.
\end{proposition}

\begin{proof}
  First of all, we shall prove that $\mathcal F_n$ is closed under
  direct limits. In view of the preceding result, we only have to see
  that $\mathcal F_n$ is closed under continuous well ordered direct
  limits. We shall make the proof inductively on $n$. Case $n=0$ is
  the well known property of flat modules. Assume that we have proved
  the result for some natural number $n$ and let us prove it for
  $n+1$. Let $(M_\alpha,\tau_{\alpha\beta})_{\alpha < \beta < \kappa}$
  be a continuous well ordered system of modules belonging to
  $\mathcal F_{n+1}$. We are going to construct, for each $\alpha <
  \kappa$, a direct system in $\mathbf{Ses}(\mathcal F_n,\mathcal F,\modr)$,
  \begin{displaymath}
    \big(0 \mapright F_\gamma \mapright^{f_\gamma} G_\gamma \mapright^{g_\gamma} M_\gamma
    \mapright 0, (\tau^1_{\delta \gamma}, \tau^2_{\delta,\gamma},
    \tau^3_{\delta\gamma})\big)_{\delta < \gamma \leq \alpha}
  \end{displaymath}
   with $\tau^3_{\gamma\alpha} = \tau_{\gamma\alpha}$ for each $\gamma \leq
  \alpha$. Then, the result will follow by taking direct limits, since
  the resulting exact sequence belongs to $\textbf{Ses}(\mathcal
  F_n,\mathcal F, \modr)$ and its last term is $M$ (note that flat dimension can be computed using flat resolutions).

  We shall make the construction by transfinite recursion on
  $\alpha$. If $\alpha = 0$ simply take a short exact sequence
  \begin{displaymath}
    0 \mapright F_0 \mapright^{f_0} G_0 \mapright^{g_0} M_0 \mapright 0
  \end{displaymath}
  with $G_0$ flat and $F_0 \in \mathcal F_n$. Now fix $\alpha > 0$ and
  assume that we have made the construction for each ordinal smaller
  than $\alpha$. If $\alpha$ is successor, say $\alpha = \gamma+1$,
  take a short exact sequence
  \begin{displaymath}
    0 \mapright F_\alpha \mapright^{f_\alpha} G_\alpha \mapright^{g_\alpha} M_\alpha \mapright 0
  \end{displaymath}
  with $g_\alpha$ a flat precover of $M_\alpha$. Then $F_\alpha$ has
  flat dimension less than or equal to $n$. Now, using
  the factorization property of the precover, we can construct the
  commutative diagram,
  \begin{displaymath}
    \commdiag{0 & \mapright & F_\gamma & \mapright^{f_\gamma} & G_\gamma &
      \mapright^{g_\gamma} & M_\gamma & \mapright & 0\cr
      & & \mapdown\lft{\tau^1_{\gamma\alpha}} & &
      \mapdown\lft{\tau^2_{\gamma\alpha}} & &
      \mapdown\rt{\tau_{\gamma\alpha}} & & \cr
      0 & \mapright & F_\alpha & \mapright^{f_\alpha} & G_\alpha &
      \mapright^{g_\alpha} & M_\alpha & \mapright & 0}
  \end{displaymath}
  Then if we set $\tau^3_{\gamma\alpha} = \tau_{\gamma\alpha}$ and
  $\tau^i_{\delta \alpha} =
  \tau^i_{\gamma\alpha}\tau^i_{\delta\gamma}$ for each $i \in
  \{1,2,3\}$ and $\delta < \gamma$, we will have successor case done. If $\alpha$ is limit the construction is made by taking
  direct limits.

  Now we prove that $\mathcal F_n$ is closed under pure submodules. Take
  a pure short exact sequence
  \begin{equation}\label{eq:1}
    0 \mapright K \mapright M \mapright L \mapright 0
  \end{equation}
  with $M \in \mathcal F_{n}$. Then there exists a direct system
  of splitting exact sequences,
  \begin{displaymath}
    \big(0 \mapright K_i \mapright^{f_i} M_i \mapright^{g_i} L_i
    \mapright 0, (\tau^1_{ij}, \tau^2_{ij},
    \tau^3_{ij})\big)_{i < i \in I},
  \end{displaymath}
  whose direct limit is (\ref{eq:1}). For any module $C$, if we apply
  $\tor_{n+1}(\_,C)$ to this system, we get another direct system of
  splitting short exact sequences whose limit is 
  \begin{displaymath}
    \setlength{\harrowlength}{20pt}
    0 \mapright \lim_{\longrightarrow}\tor_{n+1}(K_i,C) \mapright
    \lim_{\longrightarrow}\tor_{n+1}(M_i,C) \mapright \lim_{\longrightarrow}\tor_{n+1}(L_i,C) \mapright 0.
  \end{displaymath}
  Now, using that $\tor_{n+1}$ commutes with direct limits (see, for
  example, \cite[Proposition 7.8]{Rotman}), we actually get the short
  exact sequence
  \begin{displaymath}
    0 \mapright \tor_{n+1}(K,C) \mapright \tor_{n+1}(M,C) \mapright
    \tor_{n+1}(L,C) \mapright 0
  \end{displaymath}
  Finally, since $M$ has flat dimension less than or equal to $n$, we
  conclude that $\tor_{n+1}(M,C) = 0$ and, consequently,
  $\tor_{n+1}(K,C) = 0$. Since $C$
  was arbitrary, $K$ has flat dimension less than
  or equal to $n$ and the proof is finished.
\end{proof}

Note that as a consequence of \cite[Corollary 2.36]{AdamekRosicky},
this result says that, for any ring and $n \in \omega$, the full
subcategory of $\modr$ whose class of objects consists of all modules
with flat dimension less than or equal to $n$ is an accessible category.

\begin{corollary}\label{c:Definable}
  Let $n$ be a natural number and suppose that $R$ is left coherent. Then the
  class of all right $R$-modules with flat dimension less than or equal
  to $n$ is definable.
\end{corollary}

\begin{proof}
  If $R$ is left coherent then the class $\mathcal F_n$ is closed
  under direct products (this can be deduced with the same argument
  used in Proposition \ref{p:BoundedDimensionOfProducts}). Then the
  result follows from the previous proposition.
\end{proof}

\end{document}